\documentclass[3p,12pt,a4paper,twoside,fleqn,sortcompress]{amsart}
\usepackage{amssymb,latexsym,amsmath}
\usepackage{amsmath}
\usepackage{enumerate}
\theoremstyle{plain}
\newtheorem{theorem}{Theorem}[section]
\newtheorem{lemma}{Lemma}[section]
\newtheorem{corollary}{Corollary}[section]

\theoremstyle{definition}
\newtheorem{definition}{Definition}[section]
\newtheorem{example}{Example}[section]
\newtheorem{remark}{Remark}[section]

\theoremstyle{definition}

\numberwithin{equation}{section}

\newcommand{\al}{\alpha}

\newcommand{\De}{\Delta}

\newcommand{\ep}{\varepsilon}

\newcommand{\tet}{\theta}

\newcommand{\D}{\mathcal D}

\newcommand{\R}{\mathbb R}

\newcommand{\N}{\mathbf N}

\newcommand{\lp}{\left(}
\newcommand{\rp}{\right)}
\newcommand{\lsp}{\left[}
\newcommand{\rsp}{\right]}

\begin{document}
\title[ COMPACTNESS AND $\mathcal{D}$-BOUNDEDNESS IN MENGER'S 2-PROBABILISTIC NORMED SPACES]{COMPACTNESS AND $\mathcal{D}$-BOUNDEDNESS IN MENGER'S 2-PROBABILISTIC NORMED SPACES}

\author[P.K. Harikrishnan]{P.K. Harikrishnan}

\address{Department of   Mathematics,  Manipal Institute of Technology, Manipal University, Manipal-576104, Karnataka, India.}

\email{pk.harikrishnan@manipal.edu, pkharikrishnans@gmail.com }

\author[B. Lafuerza--Guill\'{e}n]{Bernardo Lafuerza--Guill\'{e}n}

\address{Departamento de Matem\'{a}tica Aplicada y Estad\'{\i}stica, Universidad de Almer\'{\i}a\\
	Al\-me\-r\'{\i}a, Spain}

\email{blafuerz@ual.es, blafuerza@gmail.com}

\author[K.T. Ravindran]{K.T. Ravindran}
\address{Department of Mathematics, Payyanur College, Kannur University, Kannur, India}

\email {drktravindran@gmail.com}
\begin{abstract}
The idea of convex sets and various related results in 2-Probabilistic
normed spaces were established in \cite{HR}. In this paper, We obtain the concepts
of convex series closedness, convex series compactness, boundedness and their
interrelationships in Menger's 2-probabilistic normed space. Finally, the idea
of $ \mathcal{D}- $ Boundedness in Menger's 2-probabilistic normed spaces and Menger's
Generalized 2-Probabilistic Normed spaces are discussed.
\end{abstract}



\maketitle

\section{Introduction and Preliminaries}

Probabilistic functional analysis has emerged as one of the important mathematical disciplines in view of its need in dealing with probabilistic models in applied problems. Probabilistic functional analysis was first initiated by Prague school of probabilistics led by Spacek and Hans in the 1950's. Probabilistic Normed spaces were introduced by \'Serstnev and its new definition was proposed by C.Alsina, B.Schwerier and A.Sklar (\cite{AlsSchw83,ASS97,SS83,SS60}). General theory of Probabilistic Metric spaces and Probabilistic Normed spaces can be read in (\cite{CI89,M, LRS99, LSG,SS60,SamSem08, SamSem09}).An important family of Probabilistic Metric spaces are Probabilistic Normed spaces. The theory of probabilistic normed spaces is important as a generalization of deterministic results of linear normed spaces and in the study of random operator equations. The concept of 2-Probabilistic normed spaces has been introduced by Fatemeh Lael and Kourosh Nourouzi (\cite{FK}).

Let $ X $ be a real linear space of dimension greater than 1. We recall the definition of a 2-norm on $X\times X$ \cite{RFY}.
\begin{definition}(\cite{RFY}) Let $ X $ be a real linear space of dimension greater than 1 and $\|.,.\|$  be a real valued function on $X\times X$ satisfying the properties, for all $ x,y,z  \in X $ and $ \alpha \in \R $
\begin{enumerate}[(N1)]
\item $\|x,y\|$ = 0 if and only if $ x $ and $ y $ are linearly dependent

\item $\|x,y\| = \|y,x\|$

\item $\|\alpha x,y\| = |\alpha|\|y,x\|$

\item $\|x+y,z\|\leq\|x,z\|+ \|y,z\|$
\end{enumerate}
then the function $\|.,.\|$ is called a 2-norm on $ X $. The pair (X,$\|.,.\|$)  is called a linear 2- normed space.
\end{definition}
It is immediate from the definition that 2-norms are non-negative and $\|x,y+\alpha x\|$ = $\|x,y\|$ for all $x$ and $y$ in $X$ and for every $\alpha$ in $\mathbb{R}$.

The most standard example for a linear 2-normed space is $X = \mathbb{R}^{2}$ equipped with the following 2-norm,
\begin{center}
$ \|x_{1},x_{2}\| = abs \left( \left\vert
                            \begin{array}{cc}
                              x_{11} & x_{12}\\
                              x_{21} & x_{22} \\
                            \end{array}
                          \right\vert \right) $
\end{center}
~~{ where}~~ $x_{i}=( x_{i1}, x_{i2})$~~{ for}~~ $i=1,2$ and "\textit{abs}" stands for absolute value of.

\begin{definition}(\cite{RFY})Let $(X,\|.,.\|)$  be a linear 2-normed space, then a map $T:X\times X\rightarrow \mathbb{R}$  is called a 2- linear functional on $X$ whenever for every $x_{1},x_{2},y_{1},y_{2} \in X $ and $\alpha,\beta \in \mathbb{R}$

 (i) $T(x_{1} + x_{2},y_{1} + y_{2}) = T(x_{1},y_{1}) + T(x_{1},y_{2})+T(x_{2},y_{1})+ T(x_{2},y_{2})$

(ii) $T(\alpha x_{1},\beta y_{1}) = \alpha \beta T(x_{1},y_{1})$

hold.
\end{definition}

Every linear 2-normed space is a locally convex topological vector space (briefly TVS). In fact, for a fixed $b\in X$, $P_{b}(x)= \|x,b\|$ for $x \in X$ is a semi norm  and the family $\{P_{b};b \in X\}$ of semi norms generates a locally convex topology on $X$.

The geometrical meaning of a 2-norm $\|x,y\|$ is that of the area of the parallelogram spanned by the vectors $ x $ and $ y $. Also, it is clear that 2-norm  $\|.,.\|$ is a continuous 2-linear functional in the linear 2-normed space $(X,\|.,.\|)$ (\cite{RFY} page 54).

 \begin{definition}(\cite{ASS97})
 A distribution function (= d.f.) is a function $F:\overline{\R}\rightarrow [0,1]$ that is non decreasing and left-continuous on $\R$; moreover, $F(-\infty)=0$ and $F(+\infty)=1$.Here $\overline{\R}=\R \cup \{-\infty,+\infty \}$. The set of all the d.f.'s will be denoted by $\Delta$ and the subset of those d.f.'s called distance d.f.'s, such that $F(0)=0$, by $\Delta^+$. We shall also consider $\mathcal{D}$ and $\mathcal{D}^+$, the subsets of $\Delta$ and $\Delta^+$, respectively, formed by the \textit{proper} d.f.'s, i.e., by those d.f.'s $F\in \Delta$ that satisfy the conditions
 $$\lim_{x\rightarrow -\infty}F(x)=0\quad \textrm{and}\quad \lim_{x\rightarrow +\infty}F(x)=1.$$
 The first of these is obviously satisfied in all of $\Delta^+$ since, in it, $F(0)=0$.

\end{definition}
For every $ a \in \R, \varepsilon_{a}$ is the d.f. defined by
$$
\varepsilon_a(t):=\begin{cases} 0, & t\leq a,\\
1,  & t>a.\end{cases}
$$
The set $ \Delta $, as well as its subsets, can partially be ordered by the usual pointwise order; in this order, $ \varepsilon_{0} $ is the maximal element in $ \Delta^{+} $.
We recall below for the reader's convenience the definition of a PN space;
\begin{definition} (\cite{AlsSchw83}, \cite{HL})\label{D:PN} A \textit{Probabilistic Normed} space is a quad\-ruple $(V,\nu,\tau,\tau^*)$,
where $V$ is a real linear space, $\tau$ and $\tau^*$ are continuous
triangle functions and the mapping $\nu:V\to\De^+$ satisfies, for all $p$ and $q$ in $V$, the conditions
\begin{enumerate}[(N1)]
\item $\nu_p=\ep_0$ if, and only if, $p=\tet$ ($\tet$ is the null vector in $V$);
\item $\forall p\in V\quad\nu_{-p}=\nu_p$;
\item $\nu_{p+q}\ge\tau\lp\nu_p,\nu_q\rp$;
\item $\forall\,\al\in\lsp 0,1\rsp\quad\nu_p\le\tau^*\lp\nu_{\al\,p},\nu_{(1-\al)\,p}\rp$.
\end{enumerate}
  If $\tau=\tau_T$ and
$\tau^*=\tau_{T^*}$ for some continuous $t$--norm $T$ and its $t$--conorm $T^*$ then $(V,\nu,\tau_T,\tau_{T^*})$ is
denoted by $(V,\nu,T)$ and is said to be a \textit{Menger} PN space. Briefly a \textit{t}-norm is any binary operation \textit{T} on $ [0,1] $ that is communicative, associative, increasing in each place and such that $ T(a,1)=a $ for every $ a\in[0,1] $. A \textit{t}-conorm  $ T^{*} $ is associated with every \textit{t}-norm $ T $; it is defined by $ T*(x,y):=1-T(1-x,1-y)$.
\end{definition}

\begin{definition}(\cite{FK})
A pair $ (X,\nu) $ is called a Menger's 2- Probabilistic Normed space (briefly Menger's 2PN space) (see \cite{SS83, RFY})  if $ X $ is a real vector space of $ dim X >1 $, $\nu$ is a mapping from $ X \times X $ into $ \D $(for each  $ x \in X $, the distribution function $ \nu(x,y) $ is denoted by $ \nu_{x,y} $ and $ \nu_{x,y}(t) $ is the value of $ \nu_{x,y} $ at  $ t \in R $ ) satisfying the axioms:
\begin{itemize}
\item[(A1)] $ \nu_{x,y}(0)=0 $ for all $ x,y \in X $

\item[(A2)] $ \nu_{x,y}(t)=1 $ for all $ t>0 $ if, and only if $x,y$ are linearly dependent.

\item[(A3)] $ \nu_{x,y}(t)=\nu_{y,x}(t) $ for all $ x,y \in X $

\item[(A4)] $ \nu_{\alpha x,y}(t)= \nu_{x,y}(\frac{t}{\vert \alpha \vert}) $ for all $ \alpha \in R\setminus\{ 0 \} $ and for all $ x,y \in X $

\item[(A5)] $ \nu_{x+y,z}(s+t) \geq min \lbrace \nu_{x,z}(s),\nu_{y,z}(t) \rbrace $ for all $ x,y,z \in X $ and $ s,t \in R $.
\end{itemize}

We call the mapping $ (x,y) \rightarrow \nu_{x,y} $  a 2-probabilistic norm (briefly 2-P norm) on $ X $.
\end{definition}
The geometrical meaning of 2-P norm on $ X $ is $\nu _{x,y} (t)= P\{ (x,y): \|x,y\| < t\}$, which is the probability of the set of all $(x,y)\in X \times X$ such that the area of the parallelogram spanned by the vectors x an y is less than t.

From the axioms A1 and A2 of the above definition, it is clear that

$$ \nu_{x,y}(t)= \varepsilon_0(t)\Leftrightarrow \textrm{$ x $ and $ y $ are linearly dependent}$$

 From a probabilistic point of view this means that for every $t>0$
 $$P\{\|x,y\| < t\}=1 \Leftrightarrow x=\lambda y, \lambda\neq 0.$$
 If one of the points $x,y$ is $\theta$ then $ x $ and $ y $ are linearly dependent and $\|x,y\|=0$.

\begin{example}(\cite{FK})
Let $ (X,\Vert .,.\Vert ) $  be a 2-normed space. Every 2-norm induces a 2-P norm on $ X $ as follows:

 $$
\nu_{x,y}(t):=\begin{cases} 0, & t\leq 0,\\
   \frac{t}{t + \Vert x,y \Vert},  & t>0.\end{cases}
   $$
This 2-P norm is called the standard 2-P norm.

\end{example}
\begin{example}(\cite{FK})
Let $(X, \|\cdot, \cdot\|)$ be a 2-normed space.  One defines for every $x,y \in X$ and $t\in\R$
the following 2-P norm
$$
\nu_{x, y}(t):=\begin{cases} 0, & t\leq \|x,y\|,\\
1,  & t>\|x,y\|.\end{cases}
$$
Then $(X,\nu)$ is a 2-PN space.
\end{example}
\begin{definition}(\cite{FK})
Let $ (X,\nu) $ be a Menger's 2-PN space, and $ (x_{n}) $ be a sequence of $ X $. Then the sequence  $ (x_{n})$ is said to be convergent to x if $ \displaystyle\lim_{n\rightarrow\infty}\nu_{x_{n}-x,z}(t) =1$, i.e. for all $ z \in X $ and $ t>0$, and $\alpha \in (0,1)$, $\exists n_0\in\N$ such that for every  $n>n_0$, one has $\nu_{x_n-x,z}(t)>1-\alpha$.
\end{definition}

\begin{definition}(\cite{FK})
Let $ (X,\nu) $ be a Menger's 2-PN space then a sequence  $ ( x_{n}) \in X $ is said to be a Cauchy sequence if $ \displaystyle\lim_{n\rightarrow\infty}\nu_{x_{m}-x_{n},z}(t) =1$ for all $ z \in X , t>0 $ and $ m>n $.
\end{definition}

\begin{definition}(\cite{FK})
A Menger's 2-PN space  is said to be complete if every Cauchy sequence in $ X $ is convergent to a point of $ X $.

	A Complete Menger's 2-PN space is called Menger's 2-P Banach space.
\end{definition}

\begin{definition}(\cite{FK}) Let $ (X,\nu) $ be a Menger's 2-PN space, $ E $ be a subset of $ X $ then the closure of $  E $ is $\overline{E}=\{x\in X: \exists (x_{n})\subset E / x_{n}\rightarrow x \rbrace$, i.e. for $ e,\in X , \alpha \in (0,1)$ and $ r>0$, $x\in \overline{E}$: there exists $n_0\in\N $ such that for every $n>n_0$ one has $\nu_{x-x_n,e}(r)\geq \alpha$.

We say, $ E $ is sequentially closed if $E=\overline{E}$.
\end{definition}

\begin{definition}(\cite{FK}) Let $  E $ be a subset of a real vector space $ X $  then $ E $ is said to be a convex set if $ \lambda x + (1-\lambda)y \in E$ for all $ x,y \in E $ and $ 0 < \lambda <1 $.
\end{definition}

\begin{definition}(\cite{FK}) Let $(X,\nu)$ be a Menger's 2-PN space, for $ e,x \in X , \alpha \in (0,1)$ and $ r>0 $ we define the locally ball by,
\begin{center}
$ B_{e,\alpha}\left[ x,r\right] =\lbrace y \in X: \nu_{x-y,e}(r)\geq \alpha \rbrace $
\end{center}
\end{definition}

\begin{definition}(\cite{FK}) Let $(X,\nu)$ and $ (Y,\nu^{'})$ be two Menger's 2-PN spaces, a mapping $ T:X \rightarrow Y $ is said to be sequentially continuous if $ x_{n}\rightarrow x $ implies $ T(x_{n})\rightarrow T(x)$.
\end{definition}

\begin{definition}(\cite{PS2008}) Let $X, Y$ be two real linear spaces of dimension greater than one and let $\nu$ be a function defined on the cartesian product  $X\times Y$ into $\Delta ^+$ satisfying the
following properties:

  \begin{enumerate}[(MG2P-N1)]
 \item $\nu_p(0)=0$ for all $(x,y) = p\in X\times Y$.
 \item $\nu_{x,y}(t)=1$ for all $t>0$ if, and only if $\nu_{x,y}=\varepsilon_0$.
 \item $\nu_{x,y}(t)=\nu_{y,x}(t)$ for all $(x,y)\in X\times Y$.
 \item $\nu_{\alpha x,y}(t)=\nu_{x, \alpha y}(t)= \nu_{x,y}\left(\frac{t}{\|\alpha\|}\right)$ for every $t>0$, $\alpha \in \R\setminus \{0\}$ \\
     and $(x,  y)\in X\times Y$.
  \item $\nu_{x + y, z}\geq \min\{\nu_{x,z}, \nu_{y,z}\}$ for every $x,y \in X$ and $z\in Y$.
  \item $\nu_{x, y + z}\geq \min\{\nu_{x,y}, \nu_{x,z}\}$ for every $x\in X$ and $y,z\in Y$.
 \end{enumerate}

\bigskip

The function $\nu$ is called a Menger generalized 2-probabilistic norm on $X\times Y$ and the pair $(X\times Y, \nu)$ is called a Menger generalized 2-probabilistic normed space (briefly MG2PN space).

\end{definition}

\begin{definition}(\cite{PS2008})
Let $A \times B$ be a non empty subset of a MG2PN space $(X\times Y,\nu)$ then its probabilistic radius $R_{A\times B}$ is defined by
$$
R_{A \times B}(x):=\begin{cases} l^{-}\varphi_{A \times B}(x), & x\in [0,+\infty),\\
1,  & x=\infty.\end{cases}
$$
where $\varphi_{A \times B}(x):= inf \lbrace \nu_{x,y}(x): x\in A,y\in B \rbrace $
\end{definition}

\begin{definition} (\cite{PS2008})
Let $A \times B$ be a non empty subset of a  MG2PN space $(X\times Y,\nu)$ then $A\times B$ is said to be:
\begin{enumerate}
\item \textit{Certainly bounded}, if $R_{A\times B}(x_{0})=1$ for some $x_{0}\in (0,\infty)$.
\item \textit{Perhaps bounded}, if one has $R_{A\times B}(x)<1$ for every $x\in (0,\infty)$ and $l^{-1}R_{A\times B}(+\infty)=1$.
\item \textit{Perhaps unbounded},if  $R_{A\times B}(x_{0})>0$ for some $x_{0}\in (0,\infty)$ and $l^{-1}R_{A\times B}(+\infty)\in (0,1)$.
\item \textit{Certainly unbounded}, if $l^{-1}R_{A\times B}(+\infty)=0$.
\end{enumerate}

$A$ is said to be $\mathcal{D}$-Bounded if either (1) or (2) holds.
\end{definition}
\begin{theorem} (see \cite{LG2001}, \cite{PS2008} )  Let $(X\times Y,\nu)$ and $A\times B$  be  a   Menger's G2PN space and a  $\mathcal{D}$-bounded subset of $X\times Y$ respectively. The set
$\alpha A\times B:=\{(\alpha p, q): p\in A, q \in B\}$ is also  $\mathcal{D}$-bounded for every fixed $\alpha \in R\setminus \{0\}$ if   $\mathcal{D}^+$ is a closed set under the t-norm $ M $.
\end{theorem}
\begin{theorem}\cite{HR} Every locally ball in Menger's 2-PN space is Convex.
\end{theorem}

\begin{theorem}\cite{HR} The closure of a closed convex set in a Menger's 2-PN space is convex
\end{theorem}

\begin{definition}\cite{HR} Let $ E $ be a  subset of a  Menger's 2-PN space $(X,\nu)$ then an element $ x \in E $ is called a interior point of $ E $ if there are $ r > 0,\,  e\in X $ such that $ B_{e,\alpha}\left[ x,r\right] \subseteq E $.

The set of all interior points of $ E $ is denoted by $ int(E) $.
\end{definition}

\begin{definition}\cite{HR} A subset $ E $ of a Menger's 2-PN space($(X,\nu)$ is said to be open if $ E=int(E) $.
\end{definition}

For any two points x, y in the real vector space $ X $ denote,

 \begin{center}
$ (x,y) = \lbrace  \lambda x+ (1-\lambda)y ; \lambda \in (0,1)\rbrace $
\end{center}

\begin{theorem}\cite{HR} Let $ E $ be a convex subset of a  Menger's 2-PN space $(X,\nu)$. Let $ a \in E $ and x is an interior point of $ E $ then every point in $ (a,x) = \lbrace  \lambda a+ (1-\lambda)x ; \lambda \in (0,1)\rbrace $ is an interior point of $ E $.
\end{theorem}

\begin{corollary}\cite{HR} Let $ E $ be a convex subset of a  Menger's 2-PN space $(X,\nu)$.Let x be an interior point of $ E $ and $ y \in \overline{E} $ then  $(x,y)\subseteq int(E) $.
\end{corollary}

\section{Main Results}

\subsection{Convex series in Menger's 2-PN space}
In this section we establish the
results that are the continuation of the convexity results in 2-probabilistic normed
spaces, obtained in the paper \cite{HR}.

\begin{definition} A subset $E$ of a Menger's 2-PN space $(X,\nu)$ is called semi closed if $E$ and $\overline{E}$ have the same interior.
\end{definition}
\begin{corollary} If the interior of a convex set $E$ of a Menger's 2-PN space $(X,\nu)$ is non-empty, then $E$ is semi closed.
\end{corollary}

\begin{proof}

It is obvious that $int (E)\subseteq \overline{E}$. Take $x \in int(E)$. If $y \in int(\overline{E})$ then by Corollary 2.1, $(x,y)\subseteq int(E) $. Let $z_\lambda = (1-\lambda)^{-1}(y-\lambda x)$ for $0 < \lambda < 1$ then as $\lambda \rightarrow 0$ we have $z_\lambda \rightarrow y$. So, $z_\lambda \in \overline{E}$ for some $\lambda$. Therefore, $y= \lambda x + (1-\lambda)z_\lambda \in int(E) $. Hence $E$ and $\overline{E}$ has the same interior.
\end{proof}

\begin{remark} If $(X,\nu^{'})$ and $(Y,\nu^{''})$ are two Menger's 2-PN spaces then $X \times Y$ equipped with a product norm
$$\nu_{[(x,y),(z,z^{'})]}(t)= min \{ \nu^{'}_{x,z}(t), \nu^{''}_{y,z^{'}}(t)\}$$
 where $[(x,y),(z,z^{'})] \in (X \times Y) \times (X \times Y)$ and $t>0$. If $(x_n)\rightarrow x$ and $(y_n)\rightarrow y$ then $(x_n, y_n )\rightarrow (x,y)$. For, $(x_n)\rightarrow x$ and $(y_n)\rightarrow y$ implies $\displaystyle\lim_{n\rightarrow\infty} \nu^{'} _{x_n -x, z} (t) = 1$ and $\displaystyle\lim_{n\rightarrow\infty} \nu^{''}_{y_n -y, z^{'}} (t) = 1$ for all $z,z^{'} \in X$.

 So, $\displaystyle\lim_{n\rightarrow\infty}[\nu_{(x_n -x, y_n -y),(z,z^�)}]= \displaystyle\lim_{n\rightarrow\infty} [min \{ \nu^{'}_{x,z}(t), \nu^{''}_{y,z^{'}}(t)\} ]=1$.
\end{remark}
\begin{lemma}
Let $(X, \|\cdot, \cdot\|)$ be a real 2-normed space. Define the standard 2-P norm
$$\nu_{x,y}(t)=\frac{t}{t + \|x, y\| },$$
where $x,y\in X$ and $t\geq 0$.Then $x_n\rightarrow x$ for the 2-norm if, and only if $x_n\rightarrow x$ for the standard 2-P norm.
\end{lemma}
\begin{proof}
Suppose $x_n\rightarrow x$ for the 2-norm $\|\cdot, \cdot\|$, $\|x_n-x,z\|\rightarrow 0$ as $n\rightarrow +\infty$ for every $z\in X $. Let $z\in X$, one has
\begin{align*}
\displaystyle\lim_{n\rightarrow \infty} \nu_{x_n-x, z}(t)&=\displaystyle\lim_{n\rightarrow \infty}\frac{t}{t+\|x_n-x,z\|}\\
&= \frac{t}{t + \displaystyle\lim_{n \rightarrow \infty}\|x_n-x,z\|}=\frac{t}{t + 0} = 1, i.e.\quad x_n\rightarrow x\quad \textrm{for the 2-P norm}.
\end{align*}
Conversely, assume that $x_n\rightarrow x$ for the 2-P norm, then

\begin{align*}
\lim_{n\rightarrow \infty} \nu_{x_n-x, z}(t) & = 1\Rightarrow \displaystyle\lim_{n\rightarrow \infty}\frac{t}{t+\|x_n-x,z\|} = 1\\
& \Rightarrow \displaystyle\lim_{n\rightarrow \infty} \|x_n-x, z\|=0\Rightarrow x_n\rightarrow x
\textrm \quad{~~ for~~ the ~~ 2-norm ~~}.
\end{align*}
\end{proof}

\subsection{Compactness and Boundedness}

\begin{definition}
A subset $E\subset X$ is said to be $compact$ if each sequence of elements of $X$ has a convergent subsequence in $ E $.
\end{definition}

\begin{definition} Let $F$ be a subset of a 2-PN space $(X,\nu)$. A convex series of elements of $F$ is a series of the form $\Sigma^{\infty}_{n=1} \lambda_{n}x_{n}$ where $x_{n}\in F$ and $\lambda_{n}\geq 0$ for each $n$ and $\Sigma^{\infty}_{n=1} \lambda_{n}=1$.

The set $F$ is said to be \textit{Convex series closed} if $F$ contains the sum of every convergent convex series of its elements. Also, $F$ is said to be \textit{Convex series compact} if every convex series of its elements is convergent to a point of $F$.
\end{definition}

\begin{lemma}
Every convex series compact set in a Menger's 2-PN space $(X,\nu)$ is Convex series closed.

\begin{proof}: Let $F$ be a convex series compact set in $(X,\nu)$ then there exists a convex series of elements of $F$, say $\sum^{\infty}_{n=1} \lambda_{n}x_{n}$ where $x_{n}\in F$ and $\lambda_{n}\geq 0$, which converges to some $x \in F$.

$\Rightarrow \displaystyle\lim_{n\rightarrow\infty}\nu_{\sum \lambda_{n}x_{n}-x ,z}(t)=1$ for all $z\in X$

$\Rightarrow \nu_{\sum \lambda_{n}x_{n}-x ,z}(t)=\varepsilon_{0}(t)$ for all $z\in X$

$\Rightarrow \sum \lambda_{n}x_{n}-x$ and $z$ are linearly dependent

$\Rightarrow \sum \lambda_{n}x_{n}-x=\lambda z$ for all $z\in X$

In particular for $\lambda z = z - x$, with $z\in F$ and $\sum \lambda_n x_n - x = (z - x) + x = z$. Hence $F$ is Convex series closed.

\end{proof}
\end{lemma}

\begin{lemma}
Let $F$ be a convex subset of a Menger's 2-PN space $(X,\nu)$ and $x_{n}\in F$ for $n \geq 1$. If $\sum \lambda_{n}=\lambda >0$ where $\lambda_{n}\geq 0$ then $\sum\lambda^{-1}\lambda_{n}x_{n}$ is a convex series of elements of $F$. So, if $\sum\lambda_{n}x_{n}\rightarrow x$ then $x=\lambda a$ where $a \in \overline{F}$.

\begin{proof}: We have $\sum\lambda^{-1}\lambda_{n}x_{n}$ is a convex series of elements of $F$ because $x_{n}\in F$ and $\lambda > 0$ with $\sum\lambda^{-1}\lambda_{n}=\lambda^{-1}\sum \lambda_{n}=\lambda^{-1}\lambda=1 $. Suppose $\sum\lambda_{n}x_{n}\rightarrow x$  then $\sum\lambda^{-1}\lambda_{n}x_{n}=\lambda^{-1}\sum \lambda_{n}x_{n}\rightarrow \lambda^{-1} x \in \overline{F}$. ie; $\lambda^{-1} x = a $ for some $a \in \overline{F}$ implies $x=\lambda a$.
\end{proof}
\end{lemma}

\begin{theorem}Let $(X,\nu)$ be a Menger's 2-PN space then every closed convex subset of $ X $ is convex series closed.
\end{theorem}
\begin{proof}
Let $F$ be a closed convex subset of $X$ and $\Sigma \lambda_{n}x_{n}$ be a convergent convex series of elements of $F$ with sum $x$. We have $\sum \lambda_{n}x_{n}=x\Rightarrow x= \lambda_{1}x_{1}+\displaystyle\sum^{\infty}_{n=2} \lambda_{n}x_{n}$. Since $\displaystyle\sum^{\infty}_{n=1} \lambda_{n}=1\Rightarrow \displaystyle\sum^{\infty}_{n=2} \lambda_{n}=1-\lambda_{1}>0$. By Remark(2.4), $x=\lambda_{1}x_{1}+(1-\lambda_{1})a$ where $a\in \overline{F}$ then $x\in F$. Hence $F$ is convex series compact.
\end{proof}
\begin{definition}A subset $F$ of a Menger's 2-PN space $(X,\nu)$ is said to be \textit{bounded} if for every $r \in (0,1)$ there exists $t_{0}>0$ such that $\nu_{x,y}(t_{0})>1-r$ for every $x \in F$ and $y\in X$.
\end{definition}

\begin{theorem} A subset $F$ of a 2-normed space $(X,\|.,.\|)$  is bounded if and only if $F$ is bounded in the Menger's 2-PN space $(X,\frac{t}{t + \|.,.\| })$.
\end{theorem}
\begin{proof}
Suppose that $F$ is a bounded subset of $(X,\|.,.\|)$ then for every $x\in F$ there exists $M>0$ such that $\|x,y\|\leq M$ for every $y\in X$. We have $\nu_{x,y}(t)=\frac{t}{t + \|x,y\| }$ for $x,y\in X$. Let $r\in (0,1)$, choose $t_{0}=\frac{M(1-r)}{r }$ then $t_{0}>0$ and $\nu_{x,y}(t_{0})=\frac{t_{0}}{t_{0} + \|x,y\| }>\frac{t_{0}}{t_{0} + M}=1-r$. So, $F$ is bounded in $(X,\frac{t}{t + \|.,.\| })$. Conversely, $F$ is bounded in $(X,\frac{t}{t + \|.,.\| })$ then for every $r \in (0,1)$ there exists $t_{0}>0$ such that $\nu_{x,y}(t_{0})>1-r$ for every $x \in F$ and $y\in X$ implies $ \frac{t_{0}}{t_{0} + \|x,y\| }> 1-r$. Choose $M=\frac{t_{0}r}{1-r}$ then $M>0$ with $\|x,y\|< M$ for every $y\in X$.
\end{proof}

\begin{theorem} Let $(X,\nu)$ be a Menger's 2-PN space and $F$ be a convex series compact subset of $X$ then
\begin{enumerate}
\item $F$ is convex series closed.
\item  $F$ is bounded.
\end{enumerate}

The converse is true if $X$ is complete.
\end{theorem}

\begin{proof}
(1) By Remark(2.3) it is clear.

(2) We prove this result by contradiction method.

Let $r\in (0,1)$ and $(a_{n}) \subset F$ such that $\nu_{a_{n},z}(2^{n})< 1-r)$ for all $n$ and $z\in X$. We have $\displaystyle\sum^{\infty}_{n=1} 2^{-n}=1$ then $\displaystyle\sum^{\infty}_{n=1} 2^{-n}a_{n}$ is a convex series of elements of $F$. Since $F$ is convex series compact, $\displaystyle\sum^{\infty}_{n=1} 2^{-n}a_{n}$ is convergent to some point in $F$. Hence $2^{-n}a_{n}$ converges to $0$ as $ n\rightarrow +\infty $ implies that for every $\epsilon >0$ and $r\in (0,1)$ there exists $k\in \N$ such that $\nu_{2^{-n}a_{n},z}(t) >1-r$ for every $n\geq k$ and $t>0$. In particular, $\nu_{2^{-n}a_{n},z}(1) >1-r$ for every $n\geq k \Rightarrow \nu_{a_{n},z}(2^{n})>1-r$, a contradiction to our assumption. So, $F$ is bounded.

Conversely, Suppose that $X$ is complete. Assume that (1) and (2) holds. One has to prove that $F$ is convex series compact. Choose $r\in (0,1)$. Since $F$ is bounded there exists $t_{0}>0$ such that $\nu_{x,y}(t_{0})>1-r$ for every $x \in F$ and $y\in X$. Let $\displaystyle\sum^{\infty}_{n=1}\lambda_{n}x_{n}$ be a convergent convex series of elements of $F$. If $\gamma_{n,m}=\displaystyle\sum^{m}_{i=n}\lambda_{i}$ then $\gamma_{n,m}\rightarrow 0$ as $n,m\rightarrow \infty$. Choose $t\in \R$ then there is $k\in \N$ such that $t\gamma^{-1}_{n,m}>0$ for every $m,n \geq k$. Since $ F $ is bounded, $ \nu_{x_{n},z}(t\gamma^{-1}_{n,m})>1-r $ implies
\begin{align*}
\nu_{\displaystyle\sum^{m}_{i=n}\lambda_{i}x_{i},z}(t)&= \nu_{\displaystyle\sum^{m}_{i=n}\lambda_{i}x_{i},z}(t\gamma^{-1}_{n,m}(\lambda_{n}+\lambda_{n+1}+...+\lambda_{m}))\\
&= \nu_{\displaystyle\sum^{m}_{i=n}\lambda_{i}x_{i},z}(t\gamma^{-1}_{n,m}\lambda_{n}+t\gamma^{-1}_{n,m}\lambda_{n+1}+...+t\gamma^{-1}_{n,m}\lambda_{m})\\
&\geq min\lbrace \nu_{\lambda_{n}x_{n},z}(t\gamma^{-1}_{n,m}\lambda_{n}),\nu_{\lambda_{n+1}x_{n+1},z}(t\gamma^{-1}_{n,m}\lambda_{n+1}),...,\nu_{\lambda_{m}x_{m},z}(t\gamma^{-1}_{n,m}\lambda_{m})\rbrace\\
&= min\lbrace \nu_{x_{n},z}(t\gamma^{-1}_{n,m}), \nu_{x_{n+1},z}(t\gamma^{-1}_{n,m}),...,\nu_{x_{m},z}(t\gamma^{-1}_{n,m}) \rbrace\\
& > min\lbrace 1-r, 1-r,..., 1-r\rbrace \\
&= 1-r
\end{align*}
ie; $\displaystyle\sum^{\infty}_{n=1}\lambda_{n}x_{n}$ is a Cauchy sequence in $ X $. So, $\displaystyle\sum^{\infty}_{n=1}\lambda_{n}x_{n}$ converges. Since $ F $ is convex series closed, the sum of $\displaystyle\sum^{\infty}_{n=1}\lambda_{n}x_{n}$ is in $ F $. Hence $ F $ is convex series compact.
\end{proof}

\begin{theorem} Let $(X,\nu)$ be a Menger's 2-PN space and $F$ be a complete, convex and  bounded subset of $ X $ then $ F $ is convex series compact.
\end{theorem}

\begin{proof}
Suppose that $\displaystyle\sum^{\infty}_{n=1}\lambda_{n}x_{n}$ is a convex series of elements of $ F $ with $ \lambda_{n}>0 $. By the same procedure in the above theorem $\displaystyle\sum^{\infty}_{n=1}\lambda_{n}x_{n}$ is a Cauchy sequence in $ X $. Take $\alpha_n= \displaystyle\sum_{i=1}^{n} \lambda_{i}$ and $ y_{n}=\displaystyle\sum^{n}_{i=1}\lambda_{i}x_{i} $. We show that $(\alpha^{-1}y_{n})$ is a Cauchy sequence in $F$. Choose $r\in (0,1)$ and $t>0$. Since $F$ is bounded there exists $t_{0}>0$ such that $\nu_{x,y}(t_{0})>1-r$ for every $x \in F$ and $y\in X$. Let $z\in X$ then we have,
\begin{align*}
\nu_{y_{n},z}(t_{0})&= \nu_{\displaystyle\sum^{n}_{i=1}\lambda_{i}x_{i},z}(\alpha_{n}t_{0}))\\
&\geq min \lbrace\nu_{\lambda_{1}x_{1},z}(\lambda_{1}t_{0}),\nu_{\lambda_{2}x_{2},z}(\lambda_{2}t_{0}),...,\nu_{\lambda_{n}x_{n},z}(\lambda_{n}t_{0}))\rbrace\\
&= min\lbrace \nu_{x_{1},z}(t_{0}), \nu_{x_{2},z}(t_{0}),...,\nu_{x_{n},z}(t_{0}) \rbrace\\
& > min\lbrace 1-r, 1-r,..., 1-r\rbrace\\
&= 1-r
\end{align*}

Since $\alpha^{-1}_{n}\rightarrow 1$ and $(y_{n})$ is a Cauchy sequence, there exists $k\in \N$ such that for all $z\in X$
\begin{align*}
\nu_{\alpha^{-1}_{n}y_{n}-\alpha^{-1}_{m}y_{m},z}(t) & = \nu_{\alpha^{-1}_{n}y_{n}-\alpha^{-1}_{m}y_{n}+\alpha^{-1}_{m}y_{n}-\alpha^{-1}_{m}y_{m},z}(\frac{t}{2}+\frac{t}{2})\\
& \geq min\lbrace\nu_{\alpha^{-1}_{n}y_{n}-\alpha^{-1}_{m}y_{n},z}(\frac{t}{2}),\nu_{\alpha^{-1}_{m}y_{n}-\alpha^{-1}_{m}y_{m},z}(\frac{t}{2})\rbrace\\
& = min \lbrace \nu_{y_{n},z}(\frac{t}{2|\alpha^{-1}_{n}-\alpha^{-1}_{m}|}),\nu_{y_{n}-y_{m},z}(\frac{\alpha_{m}t}{2}) \rbrace\\
& > min\lbrace 1-r, 1-r,..., 1-r\rbrace\\
&= 1-r
\end{align*} for every $n,m \geq k$.

Therefore, $(\alpha^{-1}_{n}y_{n})$ is a Cauchy sequence in $F$ and since $F$ is complete, $(\alpha^{-1}_{n}y_{n})$ converges to some $x \in F$. That is,there exists $k\in \N$ such that $\nu_{\alpha^{-1}_{n}y_{n}-x,z}(t)>1-r$ for every $z\in X$ and $n\geq k$ \\implies $\nu_{\alpha^{-1}_{n}\displaystyle\sum^{n}_{i=1}\lambda_{i}x_{i}-x,z}(t)>1-r$ for every $z\in X$ and $n\geq k$ as $n\rightarrow\infty$,\\ we have $\nu_{\displaystyle\sum^{n}_{i=1}\lambda_{i}x_{i}-x,z}(t)=1 $  implies  $ \displaystyle \lim_{n\rightarrow\infty}\nu_{y_{n}-x,z}(t)=1$ implies $y_{n}\rightarrow x$ and $x\in F$. Hence $F$ is convex series compact.
\end{proof}

\subsection{$\mathcal{D}$-Boundedness}
\begin{definition}
Let $A$ be a non empty subset of a Menger's 2-PN space $(X,\nu)$ then its probabilistic radius $R_{A}$ is defined by
$$
R_{A}(x):=\begin{cases} l^{-}\varphi_{A}(x), & x\in [0,+\infty),\\
1,  & x=\infty.\end{cases}
$$
where $\varphi_{A}(t):= \inf \lbrace \nu_{x,y}(t): x,y\in A\rbrace $
\end{definition}
\begin{definition}
Let $A$ be a non empty subset of a Menger's 2-PN space $(X,\nu)$ then $A$ is said to be:
\begin{enumerate}
\item \textit{Certainly bounded}, if $R_{A}(x_{0})=1$ for some $x_{0}\in (0,\infty)$.
\item \textit{Perhaps bounded}, if one has $R_{A}(x)<1$ for every $x\in (0,\infty)$ and $l^{-1}R_{A}(+\infty)=1$.
\item \textit{Perhaps unbounded},if  $R_{A}(x_{0})>0$ for some $x_{0}\in (0,\infty)$ and $l^{-1}R_{A}(+\infty)\in (0,1)$.
\item \textit{Certainly unbounded}, if $l^{-1}R_{A}(+\infty)=0$.
\end{enumerate}

$A$ is said to be $\mathcal{D}$-Bounded if either (1) or (2) holds.
\end{definition}

\begin{theorem} Let $(X,\nu)$ be a Menger's 2-PN space. If $|\alpha|\leq |\beta|$ then $\nu_{\beta x,y}(t)\leq \nu_{\alpha x,y}(t)$ for every $x,y\in X$ and $\alpha , \beta \in \R \setminus \lbrace 0\rbrace$.
\end{theorem}
\begin{proof}
We have $\nu_{\beta x,y}(t)=\nu_{ x,y}(\frac{t}{|\beta|})$ and $\nu_{\alpha x,y}(t)=\nu_{ x,y}(\frac{t}{|\alpha|})$. Since $|\alpha|\leq |\beta|$ then $\frac{t}{|\beta|}\leq \frac{t}{|\alpha|}\Rightarrow \nu_{ x,y}(\frac{t}{|\beta|})\leq \nu_{ x,y}(\frac{t}{|\alpha|})\Rightarrow \nu_{\beta x,y}(t)\leq \nu_{\alpha x,y}(t)$.
\end{proof}

\begin{theorem} Let $(X,\nu)$  and $A$ be  a Menger's 2-PN space and a nonempty subset respectively, then $A$  is $\mathcal{D}$-bounded if, and only if there exists a d.f $G\in \mathcal{D^{+}}$ such that $\nu_{x,y}\geq G$ for every $x,y\in A$.
\end{theorem}
\begin{proof}
Suppose that $A$ is $\mathcal{D}$-bounded  then there exists $R_{A}\in \mathcal{D^{+}}$. Choose $G:=R_{A}$ then $\nu_{x,y}\geq G$ for every $x,y\in A$. Conversely, Suppose that there is a d.f $G\in \mathcal{D^{+}}$ such that $\nu_{x,y}\geq G$ for every $x,y\in A$ implies $l^{-1}inf_{x,y\in A}\nu_{x,y}(t)\geq inf G(t)\Rightarrow R_{A}(t)\geq G(t)\Rightarrow \displaystyle \lim_{t\rightarrow\infty}R_{A}(t)\geq \displaystyle \lim_{t\rightarrow\infty}G(t)=1$. So, $A$ is $\mathcal{D}$-bounded.
\end{proof}

We denote the set of all $\mathcal{D}$-bounded subsets in  a  Menger's generalized 2-probabilistic normed  space $(X\times Y,\nu)$ (briefly MG2PN space) by $\mathcal{P}_{\mathcal{D}^+}(X\times Y)$.

\begin{theorem}  Let $(X\times Y,\nu)$ and $A\times B$, $C\times B$ be  a   Menger's G2PN space and two non empty $\mathcal{D}$-bounded subsets of $X\times Y$ respectively.
Then $(A + C)\times B$ is a  $\mathcal{D}$-bounded set if  $\mathcal{D}^+$ is a closed set under the t-norm $ M $, i.e. $M(\mathcal{D}^+\times \mathcal{D}^+)\subseteq \mathcal{D}^+$.
\end{theorem}
\begin{proof} For every $(a,b)\in A\times B$ and $(c,b)\in C\times B$ one has $(a + c, b)\in (A + C)\times B$. Therefore
$$\nu_{a + c, b}\geq M\{\nu_{a,b}, \nu_{c,b}\}\geq M\{\nu_{a,b}, R_{C\times B}\}\geq M\{R_{A\times B},  R_{C\times B} \},$$
and as a consequence $$R_{ (A + C)\times B}\geq M \{R_{A\times B},  R_{C\times B} \}.$$ According to hypothesis one has $$ M \{R_{A\times B},  R_{C\times B} \}\subseteq
\mathcal{D}^+,$$
and finally $\ell^-R_{ (A + C)\times B}(+\infty)=1$.
\end{proof}

\begin{theorem}  Let $(X\times Y,\nu)$ and $A\times B$, $C\times D$, $A\times D$, $C\times B$ be  a   Menger's G2PN space and four non empty $\mathcal{D}$-bounded subsets of $X\times Y$ respectively.
Then the set given by $$ A\times B + C\times D:=\{(p,q) + (r,s) = (p + r, q + s)\} $$
 is  $\mathcal{D}$-bounded if  $\mathcal{D}^+$ is a closed set under the t-norm $ M $.
\end{theorem}

\begin{proof} By ( MG2PN-3) one has, for all $(p,q)\in A\times B$, $(r,s)\in C\times D$,
$$\nu_{(p,q) + (r,s)}\geq M\{\nu_{p, q+s}, \nu_{r, q+s}\}\geq M\{R_{A\times(B + D)}, \nu_{r, q + s}\}\geq M\{R_{A\times(B + D)}, R_{C\times(B + D)}\},$$
and as a consequence
$$R_{A\times B + C\times D}\geq M\{R_{A\times(B + D)}, R_{C\times(B + D)}\}.$$
\end{proof}

\textbf{Acknowledgment:} The first author is thankful to NBHM (National Board of Higher Mathematics), Department of Atomic Energy (DAE), Govt. of India for the partial financial support and Universidad de Almer\'{\i}a, Spain for the sponsorship to visit and do this collaborative work with Prof. Bernardo Lafuerza--Guill\'{e}n. Also,the first author is thankful to MIT, Manipal University for their kind encouragement.

\end{document}